\documentclass[draft,12pt]{amsart}
\usepackage{psfrag}
\usepackage{color}
\usepackage{tikz}
\usetikzlibrary{matrix,arrows}
\usepackage{graphicx,graphics}
\usepackage{fullpage,amssymb,amsfonts,amsmath,amstext,amsthm,amscd,verbatim,enumerate}
\usepackage[T1]{fontenc}

\newtheorem{theorem}{Theorem}[section]

\newtheorem{lemma}[theorem]{Lemma}
\newtheorem{proposition}[theorem]{Proposition}
\newtheorem{corollary}[theorem]{Corollary}

\theoremstyle{definition}
\newtheorem{definition}[theorem]{Definition}

\newtheorem{example}[theorem]{Example}

\newcommand{\R}{\mathbb{R}}
\newcommand{\N}{\mathbb{N}}

\newcommand{\C}{\mathbb{C}}

\numberwithin{equation}{section}

\begin{document}

\title[]{The orbits of generalized derivatives}
\date{\today}

\author{Alastair Fletcher}
\address{Department of Mathematical Sciences, Northern Illinois University, DeKalb, IL 60115-2888. USA}
\email{fletcher@math.niu.edu}
\thanks{The first named author was supported by a grant from the Simons Foundation, \#352034}

\author{Ben Wallis}
\email{benwallis@live.com}

%
%\address{Department of Mathematical Sciences, Northern Illinois University, DeKalb, IL 60115}
%
%\email{benwallis@live.com}

\begin{abstract}
The infinitesimal space of a quasiregular mapping was introduced by Gutlyanskii et al \cite{GMRV}. Quasiregular mappings are only differentiable almost everywhere, and so the infinitesimal space generalizes the notion of derivative to this class of mappings. In this paper, we show that the infinitesimal space is either simple, that is, it consists of only one mapping, or it contains uncountable many. To achieve this, we define the orbit of a given point as its image under all elements of the infinitesimal space. We prove that this orbit is a compact and connected subset of $\R^n \setminus \{ 0 \}$ and moreover, every such set can be realized as an orbit space in dimension two. We conclude with some examples exhibiting features of orbits.
\end{abstract}
 
\maketitle
\theoremstyle{plain}
%%%%%%%%%%%%%%%%%%%%%%%%%%%%%%%%%%%%%%%%%%%%%%Introduction

\section{Introduction}

Quasiregular mappings are the natural generalization to $\R^n$, for $n\geq 2$, of holomorphic mappings in the plane. 
These are mappings with bounded distortion, and share many properties with holomorphic mappings such as value distribution and normal family results, see for example Rickman's monograph \cite{R} for much more information on this. However, they are much more flexible since they only need to be differentiable almost everywhere. 

For holomorphic maps, one only needs to look at the multiplier $f'(z_0)$ to determine the behaviour of $f$ near $z_0$. This fact plays an important role of the classification of fixed points. For quasiregular mappings, the failure of differentiability everywhere means it is a more subtle issue to describe the behaviour near a particular point.

As a generalization of the notion of a derivative for quasiregular mappings, Gutlyanskii et al \cite{GMRV} introduced generalized derivatives. The bounded distortion property and normal family machinery for quasiregular mappings mean that generalized derivatives always exist. If a mapping is differentiable at a particular point, then the generalized derivative is nothing other than a scaled version of the derivative. It is possible for a quasiregular mapping to have more than one generalized derivative at a particular point, see \cite{FMWW}. We then call the collection of all generalized derivatives at a particular point $x_0$ for a given quasiregular mapping $f$ the infinitesimal space of $f$ at $x_0$, and denote it by $T(x_0,f)$.

The main aim of the current note is to prove that the infinitesimal space either contains one element or uncountably many. The idea is to look at the orbit of a point under all elements of the infinitesimal space. We will show that this orbit arises as an accumulation set of a curve, from which the required property follows. For the converse, we also show that in the plane, every compact connected set in $\R^2 \setminus \{ 0 \}$ arises as an orbit space of a quasiconformal map $f:\R^2 \to \R^2$. We restrict ourselves here to dimension two to take advantage of computations involving the complex dilatation which are unavailable in higher dimensions. As part of the proof of this result, we construct distorted versions of logarithmic spiral maps, see Lemma \ref{lem:4}, which could conceivably be of independent interest. We refer to \cite{AIPS} and the references therein for recent work related to logarithmic spiral maps.

The paper is organised as follows. In section two, we cover preliminary material on quasiregular mappings and generalized derivatives and state our main results. In section three, we define and study properties of the orbit of a point under an infinitesimal space of a quasiregular mapping. In section four, we show that in dimension two, we can realize any non-empty, compact, connected subset of the punctured plane as an orbit space. Finally, in section five we exhibit some examples and show in particular that the behaviour of the generalized derivative on a compact set does not determine the generalized derivative.

\section{Preliminaries}

\subsection{Quasiregular mappings}

We start by recalling the definition of a quasiregular mapping.

\begin{definition}
\label{def:qr}
Let $n \geq 2$ and $U$ a domain in $\R^n$. Then a continuous mapping $f:U \to \R^n$ is called quasiregular if $f$ is in the Sobolev space $W^1_{n, loc}(U)$ and there exists $K\in [1, \infty)$ so that 
\[ |f'(x)|^n \leq K J_f(x) \: \operatorname{a.e.} \]
The smallest $K$ here is called the outer dilatation $K_O(f)$ of $f$. If $f$ is quasiregular, then it is also true that
\[ J_f(x) \leq K' \ell( f'(x) )^n \: \operatorname{a.e.}\]
for some $K' \in [1,\infty)$. Here, $\ell ( f'(x) ) = \inf_{|h|=1 } |f'(x)h|$. The smallest $K'$ for which this holds is called the inner dilatation $K_I(f)$ of $f$. The maximal dilatation is then $K(f) = \max \{ K_O(f) , K_I(f) \}$. We say that $f$ is $K$-quasiregular if $K(f) \leq K$.
\end{definition}

If $U$ is a domain in $\R^n$ with non-empty boundary, then for $x\in U$, we denote by $d(x,\partial U)$ the Euclidean distance from $x$ to $\partial U$.
One of the important properties of quasiregular mappings is that they have bounded linear distortion, which we now define.

\begin{definition}
\label{def:linear}
Let $n\geq 2$, $U \subset \R^n$ a domain, $x\in U$ and $f:U \to \R^n$ be $K$-quasiregular. For $0<r<d(x,\partial U)$, we define
\[ \ell_f(x,r) = \inf_{|y-x|=r} |f(y) - f(x) |, \: L_f(x,r) = \sup_{|y-x| = r} |f(y)-f(x)|.\]
The linear distortion of $f$ at $x$ is
\[ H(x,f) = \limsup_{r\to 0} \frac{L_f(x,r)}{\ell_f (x,r) }.\]
\end{definition}

The local index $i(x,f)$ of a quasiregular mapping $f$ at the point $x$ is
\[ i(x,f) = \inf_N \sup_y \operatorname{card} ( f^{-1}(y) \cap N ),\]
where the infimum is taken over all neighbourhoods $N$ of $x$. In particular, $f$ is locally injective  at $x$ if and only if $i(x,f) = 1$.

\begin{theorem}[Theorem II.4.3, \cite{R}]
\label{thm:lindist}
Let $n\geq 2$, $U\subset \R^n$ a domain and $f:U \to \R^n$ a non-constant quasregular mapping. Then for all $x\in U$,
\[ H(x,f) \leq C < \infty,\]
where $C$ is a constant that depends only on $n$ and the product $i(x,f)K_O(f)$.
\end{theorem}

Recall that a family $\mathcal{F}$ of $K$-quasiregular mappings defined on a domain $U\subset \R^n$ is called normal if every sequence in $\mathcal{F}$ has a subsequence which converges uniformly on compact subsets of $U$ to a $K$-quasiregular mapping.
There is a version of Montel's Theorem for quasiregular mappings due to Miniowitz.

\begin{theorem}[\cite{Mini}]
\label{thm:normal}
Let $\mathcal{F}$ be a family of $K$-quasiregular mappings defined on a domain $U \subset \R^n$. Then there exists a constant $q=q(n,K)$ so that if $a_1,\ldots, a_q$ are distinct points in $\R^n$ satisfying $f(U) \cap \{ a_1,\ldots, a_q \} = \emptyset$ for all $f\in \mathcal{F}$, then $\mathcal{F}$ is a normal family.
\end{theorem}

The constant $q$ here is called Rickman's constant and arises from Rickman's version of Picard's Theorem, see \cite[Theorem IV.2.1]{R}.

\subsection{Generalized derivatives and infinitesimal spaces}

In \cite{GMRV}, a generalization for the derivative of a quasiregular mapping $f$ at $x_0$ was defined as follows. For $t >0$, let
\begin{equation}
\label{eq:fe} 
f_{t}(x) := \frac{ f(x_0  + t x) - f(x_0) }{\rho_f(t)},
\end{equation}
where $\rho_f(r)$ is the mean radius of the image of a sphere of radius $r$ centered at $x_0$ and given by
\begin{equation}
\label{eq:rho} 
\rho_f(r) = \left(\frac{\lambda[f(B(x_0,r))]}{\lambda[B(0,1)]}\right)^{1/n} .
\end{equation}
Here $\lambda$ denotes the standard Lebesgue measure. While each $f_{t}(x)$ is only defined on a ball centered at $0$ of radius $d(x_0,\partial D) / t$, when we consider limits as $t \to 0$, we obtain mappings defined on all of $\R^n$. Of course, there is no reason for such a limit to exist, but because each $f_{t}$ is a quasiregular mapping with the same bound on the distortion, it follows from Theorem \ref{thm:lindist} and Theorem \ref{thm:normal} that for any sequence $t_k \to 0$, there is a subsequence for which we do have local uniform convergence to some non-constant quasiregular mapping.

\begin{definition}
\label{def:genderiv}
Let $f:U \to \R^n$ be a quasiregular mapping defined on a domain $U\subset \R^n$ and let $x_0 \in \R^n$. A generalized derivative $\varphi$ of $f$ at $x_0$ is defined by
\[ \varphi(x) = \lim_{k\to \infty} f_{t_k}(x),\]
for some decreasing sequence $(t_k)_{k=1}^{\infty}$, whenever the limit exists. The collection of generalized derivatives of $f$ at $x_0$ is called the infinitesimal space of $f$ at $x_0$ and is denoted by $T(x_0,f)$.
\end{definition}

To exhibit the behaviour of generalized derivatives, we consider some simple examples.

\begin{example}
Let $w\in \C \setminus \{ 0 \}$ and define $f(z) = wz$. Then it is elementary to check that if $z_0=0$, we have $f_{t}(z) = e^{i\arg w}z$ for any $t >0$. Consequently, $T(0,f)$ consists only of the map $\varphi(z) = e^{i\arg w} z$.
\end{example}

\begin{example}
Let $d\in\N$ and define $g(z) = z^d$. One can check that if $z_0 =0 $, we have $f_{t}(z) = z^d$ for any $t>0$ and so $T(0,g)$ consists only of the map $\varphi(z) = z^d$.
\end{example}

These examples illustrate the informal property the generalized derivatives maintain the shape of $f$ near $x_0$, but they lose information about the scale of $f$.
In general, if a quasiregular map $f$ is in fact differentiable at $x_0\in \R^n$, then $T(x_0,f)$ consists only of a scaled multiple of the derivative of $f$ at $x_0$, normalized to preserve the volume of the unit ball in $\R^n$. 

The reason for the scaling is the use of $\rho_f(r)$ in the definition of $f_{t}$. We may in fact replace $\rho_f(r)$ by $L_f(x_0,r), l_f(x_0,r)$ or any other quantity comparable to $\rho_f(r)$. In the special case of uniformly quasiregular mappings, that is, quasiregular mappings with a uniform bound on the distortion of the iterates, it was proved in \cite{HMM} that at fixed points with $i(x_0,f) = 1$, they are locally bi-Lipschitz. Consequently, in this special case one may replace $\rho_f(r)$ with $r$ itself. In general, quasiregular mappings are only locally H\"older continuous and so it does not suffice to use $r$ instead of $\rho_f(r)$.

\begin{definition}
\label{def:simple}
Let $f:U \to \R^n$ be quasiregular on a domain $U$ and let $x_0 \in U$. If the infinitesimal space $T(x_0,f)$ consists of only one element, then $T(x_0,f)$ is called simple.
\end{definition}

In both the examples above, the respective infinitesimal spaces are simple. It was shown in \cite{GMRV} that when the infinitesimal space is simple, then the function is well-behaved near $x_0$. In particular, 
\[ f(x) \sim f(x_0) + \rho_f(|x-x_0|) \varphi \left (\frac{x-x_0}{|x-x_0|} \right ), \]
where $\varphi $ is the unique generalized derivative in $T(x_0,f)$

\subsection{Statement of results}

Denote by $C(U,\R^n)$ the set of continuous functions from a domain $U\subset \R^n$ into $\R^n$. If $x\in U$ and $\mathcal{F} \subset C(U,\R^n)$, denote by $E_x:\mathcal{F} \to \R^n$  the point evaluation map, that is, if $f \in \mathcal{F}$ then $E_x(f) = f(x)$.

\begin{definition}
\label{def:orbit}
Let $f:U \to \R^n$ be a quasiregular mapping defined on a domain $U\subset \R^n$ and let $x_0 \in U$. Then the orbit of a point $x\in \R^n$ under the infinitesimal space $T(x_0,f)$ is defined by
\[ \mathcal{O}(x) = E_x(T(x_0,f)) = \{ \varphi(x) : \varphi \in T(x_0,f) \}.\]
\end{definition}

Our first main result relates the orbit space to the accumulation set of a curve.

\begin{theorem}
\label{thm:1}
Let $f:U \to \R^n$ be a quasiregular mapping defined on a domain $U\subset \R^n$ and let $x_0 \in U$. Then the orbit space $\mathcal{O}(x)$ is the accumulation set of the curve $t \mapsto f_t$, where $f_t$ is defined by \eqref{eq:fe}.
\end{theorem}

Moreover, \cite[Theorem 1.5]{FN} shows that $\mathcal{O}(x)$ lies in a ring $\{ x \in \R^n : 1/C' \leq |x| \leq C' \}$ for some constant $C' \geq 1$ depending only on $n, K_O(f)$ and $i(x_0,f)$.

\begin{corollary}
\label{cor:1}
Let $f:U \to \R^n$ be a quasiregular mapping defined on a domain $U\subset \R^n$ and let $x_0 \in U$. Then the infinitesimal space $T(x_0,f)$ either consists of one element or uncountably many.
\end{corollary}

Since Theorem \ref{thm:1} shows that $\mathcal{O}(x)$ is compact and connected and lies in a ring, we prove the following converse in dimension two.

\begin{theorem}
\label{thm:2}
Let $X \subset \R^2 \setminus \{ 0 \}$ be a non-empty, compact and connected set. Then there exists a quasiregular mapping $f:\R^2 \to \R^2$ for which $X$ is the image of the point evaluation map $E_{x_1} : T(0,f) \to \R^2$ for $x_1 = (1,0)$.
\end{theorem}

The proof will show that if $X \subset \{ z \in \R^2 : 1/C \leq |z| \leq C \}$, then we can bound the distortion of the quasiregular mapping in terms of $C$.

\section{The orbit space}

In this section, we prove Theorem \ref{thm:1}. To that end, let $f:U \to \R^n$ be a quasiregular mapping defined on a domain $U\subset \R^n$ and let $x_0 \in U$. By Theorem \ref{thm:lindist}, find $r_0 > 0$ small enough so that if $0<r<r_0$ then
\begin{equation}
\label{eq:1} 
\frac{L_f(x_0,r)}{ \ell_f(x_0,r)} \leq C_1,
\end{equation}
where $C_1 = 2C$ depends only on $n,K_O(f)$ and $i(x_0,f)$.
For $x\in \R^n$ fixed and $0<t \leq r_0/|x|$, consider the curve
\begin{equation}
\label{eq:2}
\gamma_x(t) = \frac{ f(x_0 +tx) - f(x_0) }{\rho_f(t) }.
\end{equation}

\begin{lemma}
\label{lem:lusin}
The curve $t\mapsto \gamma_x(t)$ is continuous for $0<t<r_0/|x|$.
\end{lemma}

\begin{proof}
Since $f$ is quasiregular, $t\mapsto f(x_0 + tx) - f(x_0)$ is continuous. It therefore suffices to show that $\rho_f$ is continuous in $t$ or, by \eqref{eq:rho}, that $\lambda[f(B(x_0,t))]$ is continuous in $t$. Suppose this is not the case and we can find $t_0 \in (0,r_0/|x|)$ and a sequence $t_k \to t_0$ so that $\lambda[f(B(x_0,t_k))] - \lambda[f(B(x_0,t_0))]$ does not converge to $0$. 
Necessarily, we must have $t_k \neq t_0$ for infinitely many $k$.
By passing to a subsequence, we can assume that $t_k - t_0$ has the same sign for all $k$ and consequently $\lambda[f(B(x_0,t_k))] - \lambda[f(B(x_0,t_0))]$ converges to $\delta > 0$. However, this means that $\lambda[ f( \partial B(x_0,t_0) ) ] > 0$. This contradicts the fact that quasiregular mappings satisfy Lusin's (N) condition, see \cite[Proposition I.4.14 (a)]{R}, that states that if $\lambda(E) = 0$, then $\lambda(f(E)) = 0$
\end{proof}

This curve $\gamma_x(t)$ is analogous to a trajectory in the theory of differential equations. 
Since $f$ is quasiregular and \eqref{eq:1} holds for $t<r_0/|x|$, it follows that $\gamma_x(t)$ is in the ring $R = \{ x: 1/C_1 \leq |x |  \leq C_1 \}$. 
Continuing the analogy with differential equations, we define the $\omega$-limit set of $\gamma_x$ as
\[ \omega(\gamma_x) = \bigcap_{s>0} \overline{ \{ \gamma_x(t) : t<s \} }.\]
We can now prove Theorem \ref{thm:1}

\begin{proof}[Proof of Theorem \ref{thm:1}]
We will show that the $\omega$-limit set of $\gamma_x$ equals $\mathcal{O}(x)$. 
First suppose that $y \in \mathcal{O}(x)$. Then we can find a decreasing sequence $t_k \to 0$ such that
\[ y = \lim_{k\to \infty} \frac{f(x_0 + t_kx)-f(x_0)}{\rho_f(t_k)}.\]
Each term in this sequence lies on $\gamma_x$ and since $t_k$ is decreasing, they are arranged in order along the curve. Hence $y$ is contained in $\omega(\gamma_x)$.

Conversely, if $y \in \omega(\gamma_x)$, then there exist $y_k \in\gamma_x$ with $y_k \to y$. By definition $y_k = \gamma_x(t_k)$ for some $t_k>0$. Since the family of functions
\[ \left \{ \frac{ f(x_0 +t_k x) - f(x_0) }{\rho(t_k) } : k\in \mathbb{N} \right \} \]
forms a normal family, there exists a subsequence converging to a generalized derivative $g$. We must have $g(x) = y$ and we are done.
\end{proof}

Since $\omega(\gamma_x)$ is closed and connected, it immediately follows from Theorem \ref{thm:1} that $\mathcal{O}(x)$ has these properties too.

\begin{proof}[Proof of Corollary \ref{cor:1}]
If for each $x\in \R^n$ the corresponding orbit space $\mathcal{O}(x)$ has one element, then the infinitesimal space $T(x_0,f)$ has only element, defined by $\varphi(x) = E_x( T(x_0,f) )$. On the other hand, if there exists $x\in \R^n$ so that $\mathcal{O}(x)$ has more than one element, then by Theorem \ref{thm:1} it has uncountably many. Consequently, there are uncountably many different generalized derivatives in $T(x_0,f)$.
\end{proof}

\section{Realizing the orbit space}

In this section, we prove Theorem \ref{thm:2}. Since we will just work in dimension $2$, we recall (see for example \cite{FM}) that if $f:\mathbb{C} \to \mathbb{C}$ is $K$-quasiconformal, then its complex dilatation is defined by
\[ \mu_f = \frac{ f_{\overline{z}}}{f_z} = \frac{ (f_x + if_y) }{(f_x-if_y)}.\]
Then $\mu \in L^{\infty}(\mathbb{C})$ with $||\mu||_{\infty} \leq k<1$ and where $k = \frac{K-1}{K+1}$. We note that if $L>0$ and $|\mu_f| = \left| \frac{L-1}{L+1} \right|$, then $f$ is $\max \{L,1/L \}$-quasiconformal.

In proving Theorem \ref{thm:2}, our construction will involve a quasiconformal map which sends circles centred at $0$ to ellipses centred at $0$ with various eccentricities and angles. Denote by $h_{K,\theta}$ the quasiconformal map obtained by stretching by a factor $K>0$ in the $x$-direction and rotating through angle $\theta$. Then
\begin{equation}
\label{eq:hkt}
h_{K,\theta}(z) = e^{i\theta} \left (  \left ( \frac{K+1}{2} \right ) z + \left ( \frac{K-1}{2} \right  ) \overline{z} \right ).
\end{equation}
The following lemma will be useful in verifying our construction has the requisite properties.

\begin{lemma}
\label{lem:3}
Let $K>0$, $\theta \in [0,2\pi)$ and let $h_{K,\theta}$ be defined by \eqref{eq:hkt}. Then for $r>0$, we have
\[ \frac{ h_{K,\theta}( r ) }{\rho(r)} = \sqrt{K} e^{i\theta}.\]
\end{lemma}

\begin{proof}
The area of the image of the disk of radius $r$ under $h_{K,\theta}$ is an ellipse with semi-axes of length $Kr$ and $r$ and hence
\[ \rho(r) = \left ( \frac{\pi K r^2}{ \pi } \right ) ^{1/2} = \sqrt{K}r.\]
Hence $\frac{ h_{K,\theta}( r ) }{\rho(r)}  = \frac{Kre^{i\theta}}{\sqrt{K}r} = \sqrt{K}e^{i\theta}$ as required.
\end{proof}

Consequently, for $h_{K,\theta}$, $x_0=0$ and $x = 1 \in \mathbb{C}$ we have for any $r>0$ that, recalling \eqref{eq:2},
\begin{equation}
\label{eq:x}
\gamma_1(r) = \sqrt{K}e^{i\theta}.
\end{equation}

We next require specific quasiconformal constructions which interpolate between $h_{K_1,\theta_1}$ and $h_{K_2,\theta_2}$ on two circles, where either $K_1 = K_2$ or $\theta_1 = \theta_2$. There are various such results in the literature, such as Sullivan's Annulus Theorem (see for example \cite{TV}), but we want an explicit interpolation for our purposes. The first is a distorted version of the logarithmic spiral map.

\begin{lemma}
\label{lem:4}
Let $K>0$. Then if $\alpha \in \R$ satisfies 
\[ |\alpha| < \left | \frac{2K}{1-K^2} \right |,\]
there exists a quasiconformal spiral map given by
\[ S(z) =  \left (  \left ( \frac{K+1}{2} \right ) z + \left ( \frac{K-1}{2} \right  ) \overline{z} \right ) |z|^{i\alpha}.\]
Moreover, we may choose $|\alpha|$ small enough so that $S$ is $2\max \{ K, 1/K \}$-quasiconformal.
\end{lemma}

As usual, $|z|^{i\alpha}$ is understood as $\exp( i\alpha \ln |z| )$ for $z\neq 0$, and $0$ otherwise. The parity of $\alpha$ indicates the direction of spiralling and $\alpha =0$ just reduces to the map $h_{K,0}$. When $K=1$, we obtain any the usual logarithmic spiral map which can take on any amount of spiralling. 

\begin{proof}
We need to check that the map $S$ defines a homeomorphism. First observe that $S$ is clearly continuous by definition.
Next, $S$ is injective on the circle of radius $t$ and maps this circle onto an ellipse. If $S$ is not globally injective, then there exist $t_1,t_2$ so that the images of the circles of radius $t_1,t_2$ under $S$ cross. Since for $t>0$ we have 
\begin{equation}
\label{eq:stz} 
S(tz) =  te^{i\alpha t} S(z),
\end{equation}
it follows that the images of the circles of radius $1$ and $t_2/t_1$ must also cross.
Consequently, it suffices to show that the condition on $\alpha$ implies that the images of the circles $|z|=1$ and $|z| = 1+t$ do not cross for all small $t>0$, and then \eqref{eq:stz} implies that $S$ is a homeomorphism.

The images of the circles of radii $1$ and $1+t$ are given by ellipses with equations
\[ \left ( \frac{x}{K} \right )^2 + y^2=1, \quad \left ( \frac{x'}{K(1+t)} \right )^2 + \left ( \frac{ y'}{1+t} \right )^2 = 1\]
respectively, where 
\[ x' = x\cos( \alpha \ln(1+t) ) - y \sin( \alpha \ln (1+t)), \quad y' = x \sin(\alpha \ln(1+t)) + y \cos(\alpha \ln(1+t)).\]
Working to first order in $t$, if $S$ is a homeomorphism then these ellipses do not intersect and hence there is no solution to the equation
\[ \left (\frac{x}{K} \right )^2 + y^2  = \left ( \frac{x-\alpha ty }{K(1+t)} \right )^2 + \left ( \frac{\alpha t x + y}{1+t} \right )^2.\]
Simplifying this equation and again neglecting terms in $t^2$, we should have no solutions to
\[ x^2+K^2y + \alpha(1-K^2)xy = 0.\]
Now this equation has no solutions when $\alpha^2(K^2-1)^2 - 4K^2 <0$. Consequently, the condition $|\alpha| < |2K|/|1-K^2|$ implies that $S$ is a homeomorphism.

To verify that $S$ is quasiconformal, we compute the complex dilatation. Since
\[ S_z = \left ( 1 + \frac{i\alpha}{2} \right )\left( \frac{K+1}{2} \right ) z^{i\alpha / 2}\overline{z}^{i\alpha /2 } + \left ( \frac{i\alpha}{2} \right ) \left ( \frac{K-1}{2} \right) z^{i\alpha/ 2 - 1}\overline{z} ^{i\alpha /2 +1},\]
and 
\[ S_{\overline{z} } = \left ( \frac{i\alpha }{2} \right ) \left ( \frac{K+1}{2} \right ) z^{i\alpha /2+1}\overline{z} ^{i\alpha / 2-1} + \left ( 1+\frac{i\alpha}{2} \right ) \left ( \frac{K-1}{2} \right ) z^{i\alpha /2}\overline{z}^{i\alpha /2},\]
we have
\[ \mu_S = \frac{S_{\overline{z}}}{S_z} = \frac{  i\alpha e^{2i\arg(z)} + \left ( \frac{K-1}{K+1}\right ) (2+i\alpha) }{2+i\alpha + i\alpha \left (\frac{K-1}{K+1} \right ) e^{-2i\arg(z)} }.\]
Writing $k = (K-1)/(K+1)$ and $\psi = 2\arg(z)$, we can rewrite this as
\begin{equation} 
\label{lem4eq1}
\mu_S = \frac{2k - \alpha \sin \psi + i\alpha (\cos \psi + k) }{ 2 + \alpha k \sin \psi + i\alpha(1+k\cos \psi )}.
\end{equation}
For $S$ to be quasiconformal, we require $||\mu_S||_{\infty} <1$. To that end, we observe that
\begin{align*}
|& 2 + \alpha k \sin \psi + i\alpha(1+k\cos \psi )|^2 - | 2k - \alpha \sin \psi + i\alpha (\cos \psi + k)  |^2 \\
&= 4(1-k^2) + 8\alpha k \sin \psi \\ &> 4[ 1-k^2 - 2 |\alpha | k] >0
\end{align*}
since the hypothesis on $\alpha$ implies that $|\alpha| < |1-k^2|/2|k|$. 
From \eqref{lem4eq1}, we see that we may choose $|\alpha|$ small enough so that the distortion of $S$ is at most $2\max \{K, 1/K\}$.
\end{proof}

The second construction involves changing the stretching in a given direction.

\begin{lemma}
\label{lem:5}
Let $K,L > 0$. Then we may find $0<t<1$ so that the map defined by
\begin{equation}
\label{eq:R} 
R(x+iy) = K(x^2+y^2)^{\nu / 2}x + iy
\end{equation}
is quasiconformal on $\{ z : t\leq |z| \leq 1 \}$, where $\nu = \ln(L/K) / \ln t$. Moreover, if $t < e^{- | \ln(L/K) |}$ then $R$ is $2\max \{ K,L,1/K,1/L \}$-quasiconformal.
\end{lemma}

The point about this construction is that on $|z| = 1$, $R = h_{K,0}$ and on $|z| = t$, $R = h_{L,0}$.

\begin{proof}
With the function $R$ as defined by \eqref{eq:R}, for $t\leq r \leq 1$, $R$ maps the circle $|z| = r$ onto the ellipse with semi-axes $ir$ and $Kr^{1+ \ln(L/K) / \ln t }$. The size of both semi-axes are continuous and monotonic in $r$ for $1+\ln(L/K) / \ln t >0$ and hence $R$ is a homeomorphism if $t < e^{- | \ln(L/K) |}$. We may therefore assume that $1+\nu >0$.

To compute the complex dilatation, the partial derivatives of $R$ are
\[ R_x = K(x^2+y^2)^{\nu/2 - 1}( (\nu +1)x^2 + y^2) \quad R_y = \nu Kxy(x^2+y^2)^{\nu/2 - 1} + i.\]
Since $R_z = (R_x - iR_y)/2$ and $R_{\overline{z}} = (R_x + iR_y)/2$, we have
\[ \mu_R = \frac{K(x^2 + y^2)^{\nu/2 - 1}( (\nu+1)x^2 + y^2 +  i \nu xy ) - 1}{ K(x^2 + y^2)^{\nu/2 - 1}( (\nu+1)x^2 + y^2 -  i \nu xy ) + 1}.\]
Clearly the absolute values of the imaginary parts of numerator and denominator agree, and the real part of the numerator is strictly less than the real part of the denominator since $1+\nu >0$. It follows that $|| \mu_R ||_{\infty} <1$ and $R$ is quasiconformal.

One can check that if $t < e^{- | \ln(L/K) |}$, that is, if $|\nu| <1$ then $| (\nu+1)x^2 + y^2 + i\nu xy | < 2(x^2+y^2)$ and consequently
\[ | \mu_R(z)| \leq \frac{ 2K |z|^{\nu} - 1 }{2K|z|^{\nu} +1 }\]
for $t\leq |z| \leq 1$. On this range, $K|z|^{\nu}$ lies between $K$ and $L$. Hence, for such a choice of $t$, $R$ is $2\max \{ K,L,1/K,1/L \}$-quasiconformal.
\end{proof}

We will now prove Theorem \ref{thm:2}, that given a non-empty, compact, connected subset of $\R^2 \setminus \{ 0 \}$, we can realize it as an orbit space for a quasiregular, and in fact quasconformal, map.

\begin{proof}[Proof of Theorem \ref{thm:2}]

Let $X \subset \R^2 \setminus \{0 \}$ be compact and connected. For $k\in \N$, let $U_k$ be an open $1/k$-neighbourhood of $X$. We can find $K\in \N$ and $C>0$ so that for $k\geq K$, $U_k \subset \{ z : 1/C \leq |z| \leq C \}$.
For $k\geq K$, find a path $\Gamma_k \subset U_k$ starting and ending at (possibly different) points of $X$ so that:
\begin{itemize}
\item $\Gamma_k$ is made up of finitely many radial line segments and circular arcs centred at $0$,
\item for every $z\in U_k$, there exists $w\in \Gamma_k$ with $|z-w| <1/k$,
\item the endpoint of $\Gamma_k$ coincides with the starting point of $\Gamma_{k+1}$.
\end{itemize}

Our aim is to construct a quasiconformal map $f$ so that, recalling \eqref{eq:2}, the curve $\gamma_{1}$ is the concatenation of $\Gamma_k$ for $k\geq K$. If this is so, then since by construction $\gamma_{1}$ accumulates exactly on $X$, we are done. Our map $f$ will send a circle of radius $r$ to an ellipse centred at $0$ with appropriate eccentricity and orientation so that $\gamma_{1}(r)$ has the required value. Recall that Lemma \ref{lem:3} and \eqref{eq:x} says what ellipse we need to obtain a required value for $\gamma_{1}(r)$.

To this end, we will give a parameterization $p_k : [r_{k+1},r_k] \to \Gamma_k$ for $k\geq K$, where $r_k$ is given and $r_{k+1}$ is to be determined.
Suppose $k\geq K$, we have the open set $U_k$ and a point $p_k(r_k) \in X$. We can find a path $\Gamma_k$ with the required properties, made up of $\Gamma_k^1, \ldots, \Gamma_k^{m}$ where $m=m(k)$ and each $\Gamma_k^j$ is either a radial line segment or an arc of a circle centred at $0$. We must have $r_k^m = r_{k+1}^1$. The parameterization for $\Gamma_k^j$ is given by $p_k^j:[r_k^{j+1},r_k^j]\to \Gamma_k^j$, where we are given $r_k^j$ and have to determine $r_k^{j+1}$.

{\bf Case (i):} $\Gamma_k^j$ is an arc of a circle, say from $se^{i\theta_1}$ to $se^{i\theta_2}$ with $1/C \leq s \leq C$ and the appropriate orientation. By \eqref{eq:x}, on $|z| = r_k^j$ we have $f(z) = h_{s^2,\theta_1}(z)$ and $\gamma_1(r_k^j) = se^{i\theta_1}$.

Apply Lemma \ref{lem:4} with $K = s^2$ and $\alpha$ chosen with parity to give the correct direction of spiralling commensurate with the orientation of our circular arc, and $|\alpha|$ chosen small enough so that $S$ has distortion at most $2\max \{ K,1/K\}$. We then choose $r_k^{j+1}$ so that on $\{ z:  r_k^{j+1} \leq  |z| \leq r_k^j \}$,
\[ f(z) = r_k^j e^{i\theta_1}S \left ( \frac{z}{r_k^j} \right ),\]
and $f(r_k^{j+1}) = s^2e^{i\theta_2}$. Then by \eqref{eq:x}, we have $\gamma_1(r_k^{j+1}) = se^{i\theta_2}$. 

{\bf Case (ii):} $\Gamma_k^j$ is a radial  line segment, say from $s_1e^{i\theta}$ to $s_2e^{i\theta}$ with $s_1,s_2 \in [1/C,C]$.
 By \eqref{eq:x}, on $|z| = r_k^j$ we have $f(z) = h_{s_1^2,\theta}(z)$ and $\gamma_1(r_k^j) = s_1e^{i\theta}$.

Apply Lemma \ref{lem:5} with $K=s_1^2$ and $L=s_2^2$ and $t$ chosen small enough that $R$ has distortion at most $2\max \{ K,L,1/K,1/L \}$. Choosing $r_k^{j+1}  = t$, we have
\[ f(z) = r_k^je^{i\theta} R \left ( \frac{z}{r_k^j} \right ),\]
and $f(r_k^{j+1}) = s_2^2e^{i\theta}$. Then by \eqref{eq:x}, we have $\gamma_1(r_k^{j+1}) = s_2e^{i\theta}$.

These two cases show how to parameterize each sub-arc of $\Gamma_k$ and hence inductively how to define a parameterization for $\gamma_{e_1}$ from $(0,r_K]$. By construction, the obtained map $f$ has distortion at most $4C^2$ and hence is quasiconformal. 
\end{proof}

We remark that if for each sub-arc, if we chose $\alpha$ and $t$ to be very small, we can obtain an upper bound for the distortion of the corresponding $f$ arbitrarily close to $C^2$.

\section{Examples}

In this section, we exhibit some examples.
\begin{enumerate}[(i)]
\item We again remark that if $f$ is simple at $x_0$, then there is only one element of the infinitesimal space and so $\mathcal{O}(x)$ consists of exactly one point for each $x \in \R^n$.
\item The logarithmic spiral map $z\mapsto z|z|^{i\alpha}$ for $\alpha \in \R$ has $\mathcal{O}(w)$ equal to the circle of radius $|w|$, for $w\neq 0$.
\item The uniformly quasiregular map $H$ constructed in \cite[Theorem 1.8]{FMWW} has a radial line segment as its orbit space for any $x\in \R^n \setminus \{ 0 \}$, since $H$ is radial and behaves like different powers of $r$ on different $r$-intervals. 
\end{enumerate}

It is worth pointing out that just because one orbit consists of one point for a given map $f$ and $x_0 \in \R^n$, it does not imply that all orbits do. For example, we can define a quasiconformal map $f:\R^2 \to \R^2$ which maps each circle of radius $r$ onto itself, fixes every point of the positive real axis but so that, for $x>0$, $f$ maps $-x$ onto $xe^{i\theta(x)}$ where $\theta$ is continuous and oscillates between $\pi/2$ and $3\pi/2$. Then for $x>0$, $\mathcal{O}(x)$ consists of one point, but $\mathcal{O}(-x)$ consists of a semicircle.

Finally, we show that the behaviour of a generalized derivative on a compact set does not determine the generalized derivative.

\begin{proposition}
There exists a quasiconformal map $f:\mathbb{D} \to \mathbb{D}$ so that for any $R>0$, there exist distinct generalized derivatives $g_1,g_2 \in T(0,f)$ which agree on $\{z:|z|<R \}$.
\end{proposition}

\begin{proof}
Let $R_t(z) = e^{it}z$ be the rotation centred at $0$ through angle $t$, and let $A$ be the ring domain $A = \{ z : 1 \leq |z|\leq 2 \}$. On $A$, define the Dehn twists
\[ f^+(z) = R_{2\pi (|z|-1)}(z) \text{ and } f^-(z) = R_{2\pi(1-|z|)}(z).\]
One can check that these mappings are bi-Lipschitz and hence quasiconformal.

Next, we define sequences $r_{n+1}<u_n<t_n<s_n<r_n$ with $r_n=1$, $r_n/s_n \to \infty$, $s_n/t_n = 2$, $t_n/u_n \to \infty$ and $u_n/r_{n+1}=2$. We define a map $f:\mathbb{D} \to \mathbb{D}$ as follows. For $|z| \in [s_n,r_n] \cup [u_n,t_n]$ we set $f(z)=z$. For $|z|\in [t_n,s_n]$, we set $f(z) = t_nf^+(z/t_n)$. Finally, for $|z| \in [r_{n+1},u_n]$ we set $f(z) = r_{n+1}f^-(z/r_{n+1})$.
The map $f$ is quasiconformal, since $f^+,f^-$ are. 

Now let $R>0$. Consider sequences $(\delta_n)_{n=1}^{\infty}$ and $(\epsilon_n)_{n=1}^{\infty}$ given by 
\[ \delta_n = \frac{r_n}{R} \text{ and } \epsilon_n = \frac{t_n}{R}.\]
Computing the generalized derivatives associated to these two sequences, we obtain $g_1,g_2$ respectively, both of which are the identity in $\{z : |z|<R\}$. This is due to the fact that $r_n/s_n \to \infty$ and $t_n/u_n \to \infty$. However, in $\{ z : R<|z| <2R \}$, $g_1$ and $g_2$ differ since they spiral in different directions in the $x_1,x_2$ plane.
\end{proof}


\begin{thebibliography}{widest-label}

\bibitem{AIPS}
K. Astala, T. Iwaniec, I. Prause, E. Saksman,
Bilipschitz and quasiconformal rotation, stretching and multifractal spectra,
{\it Publ. Math. Inst. Hautes Études Sci.}, {\bf 121} (2015), 113-154. 

\bibitem{FMWW}
A. Fletcher, D. Macclure, J. Waterman, S. Wesley, 
On the infinitesimal space of UQR mappings,
{\it Journal of Analysis}, {\bf 24}, no.1 (2016), 67-81.

\bibitem{FM}
A. Fletcher, V. Markovic,
{\it Quasiconformal mappings and Teichm\"uller theory}, OUP, 2007.

\bibitem{FN}
A. Fletcher, D. Nicks,
Superattracting fixed points of quasiregular mappings,
{\it Erg. Th. Dyn. Sys.}, {\bf 36}, no.3 (2016), 781-793.


\bibitem{GMRV}
V. Gutlyanskii, O. Martio, V. Ryazanov, M. Vuorinen,
Infinitesimal Geometry of Quasiregular Mappings,
{\it Ann. Acad. Sci. Fenn.},
{\bf 25} (2000), 101-130.

\bibitem{HMM}
A. Hinkkanen, G. Martin, V. Mayer,
Local Dynamics of Uniformly Quasiregular Mappings,
{\it Math. Scand.}, {\bf 95} (2004) no. 1, 80-100.

\bibitem{Mini}
R. Miniowitz,
Normal families of quasimeromorphic mappings, {\it Proc. Amer. Math. Soc.}, {\bf 84} (1982), 35-43.


\bibitem{R} S. Rickman,
Quasiregular mappings, Ergebnisse der Mathematik und ihrer
Grenzgebiete 26, Springer, 1993.

\bibitem{TV}
P. Tukia and J. V\"ais\"ala,
Lipschitz and quasiconformal approximation and extension,
{\it Ann. Acad. Sci. Fenn. Ser. A I Math.}, 6(2) (1982), 303-342,.

\end{thebibliography}
\end{document}